\documentclass[11pt]{article}
\usepackage[]{amsmath,amssymb,amsfonts,latexsym,amsthm,enumerate,fullpage}

\parskip=\smallskipamount
\newtheorem{thm}{Theorem}
\newtheorem*{thm*}{Theorem}
\newtheorem{lemma}[thm]{Lemma}
\newtheorem*{lemma*}{Lemma}

\newtheorem*{prop*}{Proposition}
\newtheorem{cor}[thm]{Corollary}

\theoremstyle{remark}

\newtheorem*{rmk*}{Remark}
\newtheorem*{rmks*}{Remarks}
\newtheorem*{not*}{Notation}

\newtheorem*{claim*}{Claim}
\newtheorem*{fact*}{Fact}
\newtheorem{ex}{Example}
\theoremstyle{definition}
\newtheorem{dfn}{Definition}

\begin{document}

\title{Worst Case to Average Case Reductions for Polynomials}
\author{Tali Kaufman\thanks{Research supported in part by
NSF Awards CCF-0514167 and NSF-0729011.}\\
MIT \& IAS\\
kaufmant@mit.edu\\
\and Shachar Lovett\thanks{Research supported by
ISF grant 1300/05}\\
Weizmann Institute of Science\\
shachar.lovett@weizmann.ac.il } 

\maketitle

\def \E{\mathbb E}
\def \R{\mathbb R}
\def \C{\mathbb C}
\def \Z{\mathbb Z}
\def \N{\mathbb N}
\def \F{\mathbb F}
\def \P{\mathbb P}

\def \GF{\mathcal F}

\def\<{\left<}
\def\>{\right>}
\def \({\left(}
\def \){\right)}
\def \[{\left[}
\def \]{\right]}

\begin{abstract}
A degree-$d$ polynomial $p$ in $n$ variables over a field $\F$ is
{\em equidistributed} if it takes on each of its $|\F|$ values close
to equally often, and {\em biased} otherwise. We say that $p$ has a
{\em low rank} if it can be expressed as a bounded combination of
polynomials of lower degree. Green and Tao~\cite{gt07} have shown
that bias imply low rank over large fields (i.e. for the case $d <
|\F|$). They have also conjectured that bias imply low rank
over
general fields. In this work we affirmatively answer their
conjecture. Using this result we obtain a general worst case to
average case reductions for polynomials. That is, we show that a
polynomial that can be {\em approximated} by few polynomials of
bounded degree, can be {\em computed} by few polynomials of bounded
degree. We derive some relations between our results to the
construction of pseudorandom generators, and to the question of
testing concise representations.
\end{abstract}

\newpage

\section{Introduction}\label{sec:intro}

%

Let $\F$ be a prime finite field. Let $p:\F^n \rightarrow \F$ be a
polynomial in $n$ variables over $\F$ of degree at most $d$ . We say
that $p$ is {\em equidistributed} if it takes on each of its $|\F|$
values close to equally often, and {\em biased} otherwise. We say
that $p$ has a {\em low rank} if it can be expressed as a bounded
combination of polynomials of lower degree, and {\em high rank}
otherwise. More formally we consider the following definitions.

\begin{dfn}[bias]
The {\em bias} of a function $f:\F^n \to \F$ is defined to be
$$
    bias(f) = \E_{X \in \F^n}[\omega^{f(X)}]
$$
where $\omega$ stands for the $|\F|$ root of unity, i.e. $\omega =
e^{\frac{2\pi}{|\F|}}$
\end{dfn}

\begin{dfn}[rank] Let $p(X)$ be a degree $d$ polynomial over $\F^n$.
$rank_{d-1}(P)$ is the smallest integer $k$ such that there exist
degree $d-1$ polynomials $q_1(X),...,q_k(x)$, and a function $F:\F^k
\to \F$, s.t. $p(X) = F(q_1(X),...,q_k(X))$.
\end{dfn}

Green and Tao \cite{gt07} have shown that over large fields bias
imply low rank.

\begin{thm}[Theorem 1.7 in \cite{gt07}]\label{thm:gt}
Let $p(X)$ be a degree $d$ polynomial over $\F^n$, where $d<|\F|$.
If $bias(p) \ge \delta > 0$, then
$rank_{d-1}(p) \le c(\delta,d)$.
\end{thm}

In their paper, Green and Tao conjecture that the restriction $d <
|\F|$ can be removed, but their proof technique breaks down when $d
\ge |\F|$. Note that over large
fields things might behave
differently than over small fields. One important example is the
%
The {\em Inverse Conjecture for the Gowers Norm}. This conjecture
roughly says that if the $d$-derivative of a polynomial is biased
then that polynomial has a non-negligible correlation with some
polynomial of degree
$d-1$. The {\em Inverse Conjecture for the Gowers Norm} was proven
to be true over large fields by~\cite{gt07}, but was proven to be
false over small fields~\cite{gt07,lms}. One of the main tools used
for proving the conjecture over large fields was
Theorem~\ref{thm:gt}, that was proven over large
fields.

One could ask what is the case with the above theorem, whether it
remains true over smaller fields or it becomes false there. We show
that the~\cite{gt07} result is true over general fields. In this
respect, as opposed to the {\em Inverse Conjecture for the Gowers
Norm} case,
large and small fields behave similarly.

\subsection{Our Main Results}
Our first main theorem shows that bias imply low rank over general
fields.

\begin{thm} [Bias imply low rank for general fields ]\label{thm:main}
Let $p(X)$ be a degree $d$ polynomial over $\F^n$, s.t. $bias(p) \ge
\delta > 0$. Then
$rank_{d-1}(p) \le c(d,\delta)$. That is, there exist
degree-$(d-1)$ polynomials $q_1(X),...,q_c(x)$, and a function
$F:\F^c \to \F$, s.t. $p(X) = F(q_1(X),...,q_c(X))$, and
$c=c(d,\delta)$. Moreover, $q_1,...,q_c$ are derivatives of the form
$p(X+a)-p(X)$ where $a \in \F^n$.
\end{thm}

Most of the technical part of the paper is dedicated to proving
Theorem~\ref{thm:main}. The proof will go by induction on the degree
$d$ of $p(X)$. Notice that for $d=1$ it holds trivially. So, we
assume Theorem~\ref{thm:main} to hold for all degrees smaller than
$d$, and prove it for degree $d$.

Our second main theorem is obtained as a corollary from
Theorem~\ref{thm:main}. The theorem is a worst case to average case
reduction for polynomials. It says that a polynomial that can be
approximated by few polynomials of bounded degree, can be computed
by few polynomials of bounded degree. We now move to define and
prove this rigorously.

\begin{dfn} [$\delta$-approximation]
We say a function $f:\F^n \to \F$ $\delta$-approximates $p(X)$ if:
$$
|\E_{X \in \F^n}[\omega^{p(X)-f(X)}]| \ge \delta
$$
\end{dfn}

\begin{thm}[Worst-case to average case reduction for polynomials of bounded degree]\label{thm:avgtoworstpolys}
Let $p(X)$ be a polynomial of degree $d$, $g_1,...,g_c$ polynomials
of degree $k$, ($d,c,k = O(1)$) and $F:\F^c \to \F$ a function s.t.
the composition $G(x) = F(g_1(X),...,g_c(X))$ $\delta$-approximates
$p$. Then there exist $c'$ polynomials $h_1,...,h_{c'}$ and a
function $F':\F^{c'} \to \F$ s.t.
$$
F'(h_1(X),...,h_{c'}(X)) \equiv p(X)
$$
Moreover, $c'=c'(d,c,k)$ (i.e. independent of $n$) and each $h_i$ if
of the form $p(X+a)-p(X)$ or $g_j(X+a)$ for $a \in \F^n$.
In particular, if $k \le d-1$ then also $\deg(h_i) \le d-1$.
\end{thm}

\begin{proof}
Develop $\omega^{F(z_1,...,z_c)}:\F^c \to \C$ in the Fourier basis.
If $F(g_1(X),...,g_c(X))$ $\delta$-approximates $p(X)$, there must
exist some Fourier coefficient which $\delta'$-approximates $p$
($\delta' \ge \delta |\F|^{-c}$). That means, there exist
$\alpha_1,...,\alpha_c \in \F$ s.t. the polynomial
$$
p'(x) = p(x) - (\alpha_1 g_1(x) + ... \alpha_c g_c(x))
$$
has bias at least $\delta'$. Using Theorem~\ref{thm:main} we get
that there must exist at most $c'$ derivatives
of $p'$ which computes it. We can now use them and $\alpha_1 g_1 +
... \alpha_c g_c$ to compute $p$.
\end{proof}

\section{Significance of Results}\label{sec:significance}

\paragraph{Bias imply low rank over general fields.}
Out first main theorem
(Theorem~\ref{thm:main}) shows that over general
fields there is a phenomena that bias imply low rank. Green and Tao
\cite{gt07} proved this for large fields. They left the case of
small fields open. We answer their question affirmatively by showing
that the "bias imply low rank" phenomena is robust and holds for all
fields.

\paragraph{Worst case to average case reductions for polynomials.}

Our second main theorem
(Theorem~\ref{thm:avgtoworstpolys}) shows that every polynomial,
not necessarily biased, that is approximated by few other bounded
degree polynomials, can be
computed by few bounded degree polynomials. We view this result as a
worst case to average case reduction for polynomials. I.e. in order
to show that a polynomial can not be approximated by few bounded
degree polynomials, it would be sufficient to show that
the polynomial
can not be
calculated by few bounded degree
polynomials. That later task might be easier. An example when such a
scenario is relevant is the following. The
papers~\cite{gt07,lms} that disprove the
{\em Inverse Conjecture for the Gowers Norm} needed to show that the symmetric polynomial $S_4$ over
$\F_2$, i.e. $S_4(x_1,...,x_n) = \sum_{i<j<k<l} x_i x_j x_k x_l$ can
not be approximated by few degree $3$ polynomials. Given the current
result it could be sufficient (and maybe easier?) to show that $S_4$
can not be computed by few degree $3$ polynomials.


\paragraph{On the power of induction and relation to pseudorandom generators.}
Pseudorandom generator for polynomials of degree-$d$ is an efficient
procedure that stretches $s$
field elements into $n\gg s$ field elements that can fool any
polynomial of degree $d$ in $n$ variables. Pseudorandom generators
are mostly interesting over small fields. One can use our first main
theorem to provide an alternative proof to the correctness of the
pseudorandom generators of~\cite{bv} that fool degree $d$
polynomials. The argument used by~\cite{bv} relied on the {\em Gower
Inverse Conjecture} which turned out to be false for small
fields~\cite{gt07,lms}. However, a better inspection of
the~\cite{bv} argument shows that for proving the correctness of
their pseudorandom generators one can only rely on the statement
that "bias imply low rank" which we prove here.

The "bias imply low rank" idea suggests a robust way to
construct pseudorandom generators for some complex function classes
based on pseudorandom generators for simpler function classes. The
last would be done by the following methodology. Either you are
unbiased in which case you could fool whoever you wanted to fool, or
you are a function of few functions of lower complexity, so by {\em
induction} we obtain a construction of pseudorandom generator for
functions of higher complexity classes (e.g. degree $d$ polynomials)
given pseudorandom generators for functions of lower complexity
classes (e.g. linear functions).

\paragraph{Relation to testing concise representations.}
Diakonikolas et al.~\cite{r} suggest a general methodology to test
whether a function on $n$ variables has a concise representation.
The idea is to do testing by implicit learning. Their work provides
property testers for several concise structures among them are
$s$-sparse polynomials, size-$s$ algebraic circuits and more.
Consider the following concise representation of degree $d$
polynomials. A polynomial of degree $d$ has a concise representation
if it is a function of few polynomials of lower degree (i.e. if it
has a low rank). We argue that one can use the "bias imply low rank"
theorem in order to construct a tester that test for this concise
representation. The tester first performs a low degree testing e.g.
by~\cite{rs} to test that the given polynomial is of degree at most
$d$, if the degree-tester rejects the tester rejects otherwise, the
tester would approximate the bias of the polynomial. If the bias is
large then by our first theorem the polynomial has low rank and the
tester accepts, otherwise it rejects. The idea behind our approach
for testing this concise representation of polynomials is robust in
the following sense. It suggests a methodology for testing concise
representation of some family (e.g depth $d$ circuits) given a
membership tester for that family, and given that the family obeys
the "bias imply low rank" principle. If these two conditions are
met, one can construct a tester. The tester first test membership in
the family and then estimate the bias. If the bias is high the rank
is low and concise representation exists.


\paragraph{Extension to tensors}
Let $L(x,y)$ be a bilinear form over $\F^n$, i.e. a function of the form
$$
L(x,y) = x^t A y
$$
where $x,y \in \F^n$ and $A$ is a matrix.
There is a close connection between the rank of the matrix and the bias
of $L$. Dixon's Theorem
(\cite{ms})
tells us that the bias of $L$ (and in fact, all non-zero Fourier coefficients
of $L$) has absolute value $c(\F)^{-rank(A+A^t)}$.
The theory of higher dimensional multilinear forms, i.e. tensors, is much
less understood. In particular, there is no single notion of tensor rank.
We prove, as a direct corollary of Theorem~\ref{thm:main},
that if we define the rank of a tensor as minimal number of lower degree
multilinear forms needed to compute it, then bias imply low rank for tensors.
\begin{thm}\label{thm:tensorrank}
Let $L(X_1,...,X_d)$ be a multilinear form of degree $d$ s.t. $bias(L) \ge \delta > 0$.
Then, there exist degree-$(d-1)$ multilinear forms $q_1,...,q_c$, each operating on
$d-1$ variables out of $X_1,...,X_d$, and a function
$F:\F^c \to \F$, s.t.
$$L(X_1,...,X_d) = F(q_1(X_1,...,X_{t_1-1},X_{t_1+1},...,X_d),...,q_c(X_1,...,X_{t_c-1},X_{t_c+1},...,X_d))
$$
and $c=c(d,\delta)$. Moreover, $q_1,...,q_c$ are derivatives of $L$.
\end{thm}
\begin{proof}
We use Theorem~\ref{thm:main} on $L$ as a degree $d$-polynomial, and observe
that derivatives of $L$ are sums of $d$ degree-$(d-1)$ multilinear forms in $d-1$ variables of $X_1,...,X_d$.
\end{proof}

\subsection{Proof Overview}
The proof starts by a lemma of Bogdanov and Viola, showing that if a degree-$d$ polynomial
$p(X)$ has bias, then we can build a constant-size circuit which approximates it, whose inputs are degree-$(d-1)$
polynomials (and in fact derivatives of $p$).

The technical heart of the paper is the proof of the following
statement (Lemma~\ref{lemma:computation}): A biased polynomial of
degree $d$ that is approximated by few degree $d-1$ polynomials can
be computed by few degree $d-1$ polynomials.

In general our proof structure is similar in spirit to that
of~\cite{gt07}, however, there is a clear distinction between the
two approaches that enable us to obtain the stronger result. The
proof is by induction on the degree. An important notion in the
proof is the following definition of a factor.

\begin{dfn}[Factor] A factor is a set of polynomials
$g_1,...,g_m:\F^n \rightarrow \F$.
The number of polynomials in the set $m$ is the {\em dimension} of
the factor. The maximum degree
is the {\em order} of the factor. The polynomials of the factor
divide the hyper-cube into $|\F|^m$ parts according to their joint
image. Each such part is called a {\em region}.
\end{dfn}

\paragraph{All regions are of almost same size.}
The first step of the proof shows that given a set of polynomials of
degree at most $d-1$ that approximates a degree-$d$ polynomial $p$,
i.e. given a factor that approximates $p$, one can transform that
factor into another factor of constant size that approximate $p$ in
in which all of the regions are roughly of the same size. In order
to obtain that we need the following definition of regularity.

\begin{dfn} [Regularity - informal]
A factor is regular if the joint distribution of its polynomials is
close to uniform. see formal definition in Definition~\ref{dfn:regularity}.
\end{dfn}

The Regularity Lemma (see Lemma~\ref{lemma:regularity}) shows that
given a factor of constant dimension that approximate $p$, it can be
transformed into another constant dimension factor which is regular
and approximate $p$. Moreover, in a regular factor all regions are
roughly of the same size (see Lemma~\ref{lemma:regionsuniform}).

\paragraph{Averaging arguments.}
Since we know that all regions are roughly of same size we can use
averaging arguments to claim that most regions have large agreement
with the polynomial
$p$. These are denoted as {\em almost good regions}. The rest of the
regions are denoted as {\em bad regions} but there are only few of
them. We then show that almost good regions are good (i.e. they
fully agree with $p$, see
Lemma~\ref{lemma:noalmostconstantregions}). We further show that $p$
must be fixed/constant over bad regions
(Lemma~\ref{lemma:allregionsconstant}). Hence we get that the
polynomial $p$ is a function of the factor, and can be computed (and
not only approximated) by the factor. So the heart of the proof is
to show that almost good regions are good and that $p$ is constant
on the few bad regions.

\paragraph{Almost good regions are good, and $p$ is fixed on the rest - a wishful scenario.}
For a set of variables $Y_1,Y_2,... \in \F^n$ we denote by
$Y_I = \sum_{i \in I}Y_i$. Since $p(x)$ is a polynomial of degree
$d$ it
satisfies cube constraints of the following form:
$$
p(x) = \sum_{I \subseteq [d+1], |I|>0} (-1)^{|I|+1} p(x + y_I)
$$

We show that for every point $x$ that belongs to an {\em almost good
region}, there is a constraint that pass through it and all its
other points (i.e. points of the form $x + y_I$ for $|I|>0$) belong
to the good part of the same region. As all the values $p(x +
y_I)=c$ for all good points in the region, we get that also
$p(x)=c$. Hence we get that the value $p(x)$ is constant on the
region for every $x$.

A question of interest here is given $x$ what is the probability
that $x+y_I$ is in the region of $x$ for every $I$. Since the
assignment of points to the region is determined by the values of
the polynomials
$g_1, \cdots, g_m$ that compose the factor, saying that for every
$I$, $x+y_I$ is in the region of $x$ is equivalent to the following
condition.
$$[g_i(x)=g_i(x+y_I) \mbox{ for all } i \in [m] \mbox{ and } I \subseteq [d+1]]$$

One can observe that for this condition to hold, due to the
dependence between derivatives, it is sufficient to require the
following.

$$[g_i(x)=g_i(x+y_I) \mbox { for all }i \in [m] \mbox{ and }I \subseteq [d+1] \mbox{
s.t. }  1 \leq |I| \leq \deg(g_i)]$$

Note that this condition by itself is not sufficient to ensure that
$p(x)$ is assigned the value of the good part of the region. One
need also to add the requirement that non of the $x+y_I$ fall in the
bad part of the almost good region. Once making
the calculations (Lemma~\ref{cor:jointprob}) one can realize that if
all the events that compose the above conditions were independent
then $p(x)$ would have get the correct value. Hence we could have
say that all the almost good regions are totally good
(Lemma~\ref{lemma:noalmostconstantregions}). So, $p$ agrees with the
factor on all the regions but few. Given this, we show that $p$ is
fixed also on the bad regions (Lemma~\ref{lemma:allregionsconstant})
However, the above arguments work under a wishful assumption that
all the considered events are independent. Much of the technical
effort of this work goes into showing that the joint distribution of
these events is close to being uniform, i.e. the events are almost
independent.

\paragraph{Obtaining almost independence through interpolation.}
One way to prove the almost-independent argument is to relate the bias
from independence to the bias of the $d$-derivatives of the
$g_i$'s that compose the factor. Using interpolation one can relate
the bias of the $d$-derivative of the
$g_i$'s to the bias of the
$g_i$ which should be small by the assumption about the regularity
of the factor. The use of interpolation in the above argument is
absolutely crucial, and that what allow~\cite{gt07} to get a result
only for for the case $d < |\F|$.

\paragraph{Our approach for dealing with general fields.}
In order to eliminate the need for interpolation that could hold
{\bf only} over large fields and in order to be able to claim that
$p$ is indeed fixed on bad regions we define a stronger notion of
regularity. The stronger regularity roughly requires uniformity of
the $d$-derivatives of the polynomials in the factor

\begin{dfn} [Strong Regularity - informal]
A factor is strongly regular if the joint distribution of the
$d$-derivatives its polynomials is close to uniform. see formal
definition in Definition~\ref{dfn:strongregularity}.
\end{dfn}

Based on this stronger regularity we prove counting lemmas
(Lemmas~\ref{lemma:belowdeltainducesall} and
~\ref{lemma:belowdeltaindependent}) that enable us to get almost
independence in the above sense without the need for interpolation.
Thus, we obtain our results for general fields.
In the following we discuss the usefulness of strong regularity.
Using the strong regularity we allow all polynomial $g_1,...,g_m$ to
participate in the calculation of $g_i(x+Y_I)$, in contrast to the
definition of Green and Tao which required only evaluations of the
same $g_i$. This gives raise to the notion of an "independence
degree" of a polynomial $\Delta(g_i)$ (which is between $1$ and
$\deg(g_i)$, and could be strictly lower than $\deg(g_i)$ given the
other derivatives in the factor). Thus, instead of requiring a
uniform joint distribution over the following set of events as Green
and Tao do:
$$[g_i(x)=g_i(x+y_I) \mbox { for all }i \in [m] \mbox{ and }I \subseteq [d+1] \mbox{
s.t. }  1 \leq |I| \leq \deg(g_i)]$$

We could only required uniform joint distribution over a subset of
the events, that is over:

$$[g_i(x)=g_i(x+y_I) \mbox { for all }i \in [m] \mbox{ and }I \subseteq [d+1] \mbox{
s.t. }  1 \leq |I| \leq \Delta(g_i)]$$

It turns out that this stronger notion of independence allows us to
get almost independence between the variables, without the need of
integration, which makes our result work for every field.

In the following we show an example that $\Delta(g_i)$ could be
strictly smaller then $\deg(g_i)$.

\begin{ex}
Consider the symmetric polynomial $S_4$ over $\F_2$, i.e.
$$
S_4(x_1,...,x_n) = \sum_{i<j<k<l} x_i x_j x_k x_l
$$
Consider the fourth derivative of $S_4$, i.e. the polynomial in
$X,Y_1,...,Y_4$
$$
G(X,Y_1,...,Y_4) = \sum_{I \subseteq [4]} S_4(X+Y_I)
$$
This polynomial corresponds to the $4$-th Gowers Norm of $S_4$, and
as was shown in $\cite{gt07}$ and $\cite{lms}$, it has bias $1/8$.
In particular, it cannot be independent (and so we would have
defined $\Delta(S_4)$ to be at most $3$).
\end{ex}

The strong regularity for polynomials that we define here might find
some future applications.

\subsection{Organization}
The rest of the paper is organized as follows. We define required
notation in Section~\ref{sec:prelim}. We define and analyze regular
and strongly regular factors in Section~\ref{sec:regularity}. We show
that strong regularity implies almost independence in Section~\ref{sec:almost_ind_by_reg}. We prove
Theorem~\ref{thm:main} in Section~\ref{sec:approxtocalc}.

\section{Preliminaries}\label{sec:prelim}
$\F$ if a fixed prime field. We work with constant degree
polynomials over $\F^n$. We denote by capital letters $X,Y,...$
variables in $\F^n$, and by small letters $x,y,a,...$ values in
$\F^n$. Degree of a polynomial will always mean total degree. Unless
otherwise specified, when we speak of a degree $d$ polynomial, we
mean in fact a polynomial of total degree at most $d$. For a set of
variables $Y_1,Y_2,... \in \F^n$ we denote by $Y_I = \sum_{i \in
I}Y_i$, and similarly for a set of values $y_1,y_2,... \in \F^n$. We
write $u = v (1 \pm \epsilon)$ for $u \in [v(1-\epsilon),
v(1+\epsilon)]$. When we speak of a {\em growth function}, we mean
any monotone function $\GF:\N \to \N$ (for example, $\GF(n) =
2^{n^2}$).

\begin{dfn}[Derivative space of a polynomial]
For a polynomial $f(X)$, we define its derivative space to be the
set
$$
Der(f) = \{f(X+a)-f(X): a \in \F^n\}
$$
\end{dfn}
Notice that if $\deg(f)=k$ then all polynomials in $Der(f)$ have
degree at most $k-1$.

\section{Regularity of polynomials}\label{sec:regularity}
As we discussed in the introduction, the notion of regularity plays
a major rule in our proof. Green and Tao in~\cite{gt07} suggested
one notion of regularity (we refer to it henceforth as {\em
regularity}) which limited their proof to work only for large fields
(i.e. $d <|\F|$). We suggest a stronger notion of regularity (noted
henceforth as {\em strong regularity}). This new notion of strong
regularity is essential for obtaining a result for general fields.
In the following we review the regularity definitions given by Green
and Tao. Then, we present the notion of strong regularity and show
that every constant factor that approximates a polynomial $p$ can be
transformed into a constant factor that approximates $p$ and is also
strongly regular. We end this section by showing that strong
regularity implies almost independence for sets of variables that
forms some specific structures. This almost independence is the crux
of the proof of Theorem~\ref{thm:main}.


\begin{dfn}[Regularity of polynomials]\label{dfn:regularity}
Let $\GF$ be any growth function. A set of polynomials
$\{g_1,...,g_m\}$ is called $\GF$-regular if
any linear
combination $\alpha_1 g_1(X) + ... \alpha_m g_m(X)$ cannot be
expressed as a function of at most $\GF(m)$ polynomials of degree
$k-1$, where $k = \max\{\deg(g_i): \alpha_i \ne 0\}$ (i.e. $k$ is
the maximal degree of $g_i$ appearing in the linear combination).
\end{dfn}

Green and Tao also define the notion of a refinement of a set of
polynomials. Informally, a set $\{g_1,...,g_m\}$ is a refinement of
$\{f_1,...,f_s\}$ if for any $i \in [s]$, $f_i(x)$ can be computed
given the values of
$\{g_1(x),...,g_m(x)\}$.

\begin{dfn}[Refinement]\label{dfn:refinement}
A set of polynomials $\{g_1,...,g_m\}$ is a refinement of
$\{f_1,...,f_s\}$ if for any $i \in [s]$ there exists a function
$F_i:\F^m \to \F$ s.t.
$$
f_i(X) = F_i(g_1(X),...,g_m(X))
$$
\end{dfn}

Green and Tao prove that for any growth function $\GF$, any set of
polynomials
$F=\{f_1,...,f_s\}$ can be refined to a $\GF$-regular set
$\{g_1,..,g_m\}$, s.t. $m$ depends only on $s$, $\GF$ and the
maximal degree
in $F$. Importantly, $m$ is independent of $n$. Green and Tao proof
start by a set $\{f_1,...,f_s\}$ which approximates $p(X)$ weakly,
transform it to a set which approximates $p$ on almost all points,
then use the regularity condition to show that it must in fact
compute $p$ exactly. In order to prove this, they need to analyze
the joint distribution of
$$
\{g_i(X + \sum_{t \in I} Y_t): i \in [m], I \subseteq [D]\}
$$
where $D=O(d)$, $X,Y_1,...,Y_D \in \F^n$ are independent variables.
Lets denote by $Y_I = \sum_{i \in I} Y_i$. They prove that if we
just look on the subset
$$
\{g_i(x + Y_I): i \in [m], I \subseteq [D], |I| \le \deg(g_i)\}
$$
for any $x \in \F^n$, then these variables must be almost
independent, and for any $|I| > \deg(g_i)$, $g_i(x+Y_I)$ is
determined by $\{g_i(X+Y_J): J \subseteq I,\ |J| \le \deg(g_i)\}$.
Since the regularity requirement was for evaluations $\{g_i(X): i
\in [m]\}$ they needed to use integration over $\F$ to get the
independence result for evaluation on hypercubes, which limited
their proof only to $d < char(\F)$.

We follow a similar approach, but in order to allow for $d \ge
char(\F)$, we allow more freedom in the set of variables which are
almost independent or fixed given the others. For a set of
polynomials $g_1,...,g_m$, we also have a "independence degree"
$\Delta$ for every $g_i$ (which is between $1$ and $\deg(g_i)$).
Instead of requiring as Green and Tao do that:
$$
\{g_i(x + Y_I): i \in [m], I \subseteq [D], |I| \le \deg(g_j)\}
$$
are almost independent, we demand only that
$$
\{g_i(x + Y_I): i \in [m], I \subseteq [D], |I| \le \Delta(g_j)\}
$$
are almost independent. However, we also demand that for any $i \in
[m]$ and $|I| > \Delta(g_i)$, the value of $g_i(x + Y_I)$ can be
determined by $\{g_j(x + Y_J): J \subseteq I,\ |J| \le
\Delta(g_j)\}$. Notice that we allow all polynomial $g_1,...,g_m$ to
participate in the calculation of $g_i(x+Y_I)$, in contrast to the
definition of Green and Tao which required only evaluations of the
same $g_i$. It turns out that this stronger notion of independence
allows
us to get almost independence between the variables, without the
need of integration, which makes our result work for every field.

We now move to formally define our notion of strong regularity, and
to show it implies the almost independence/total dependence
structure we have just described. We first extend the definition of
a derivative space to several polynomials in several variable sets.

\begin{dfn}[Derivative space]
For a set of polynomials $F=\{f_1(X),...,f_s(X)\}$ we define:
$$
Der(F) = \{f_i(X+a)-f_i(X): i \in [s],\ a \in \F^n\}
$$
Similarly, for a set of polynomials in several variables
$F=\{f_1(Y_1,...,Y_k),...,f_s(Y_1,...,Y_k)\}$ ($Y_1,...,Y_k \in
\F^n$) we define:
$$
Der(F) = \{f_i(Y_1+a_1,...,Y_k+a_k)-f_i(Y_1,...,Y_k): i \in [s],\
a_1,...,a_k \in \F^n\}
$$
\end{dfn}

Notice that if the maximal degree of polynomials in $F$ is $k$, then
the maximal degree of polynomials in $Der(F)$ is at most $k-1$. We
now define strong regularity.

\begin{dfn}[Strong regularity of polynomials]\label{dfn:strongregularity}
Let $\GF$ be any growth function. Let $G=\{g_1,...,g_m\}$ be a set
of polynomials and $\Delta:G \to \N$ be a mapping from $G$ to the
natural numbers. We say the set $G$ is strong $\GF$-regular with the
degree bound $\Delta$ if:
\begin{enumerate}
  \item For any $i \in [m]$, $1 \le \Delta(g_i) \le \deg(g_i)$.

  \item For any $i \in [m]$ and $r > \Delta(g_i)$, let $X$ and $Y_1,Y_2,...,Y_r$ be variables in $\F^n$. There exist a function $F_{i,r}$ s.t.
      $$
      g_i(X + Y_{[r]}) = F_{i,r}\(g_j(X+Y_J): j \in [m],\ J \subseteq [r],\ |J| \le \Delta(g_j)\)
      $$

  \item For any $r \ge 0$, let $X$ and $Y_1,...,Y_r$ be variables in $\F^n$.
      Let $\{\alpha_{i,I}\}_{i \in [m], I \subseteq [r], |I| \le \Delta(g_i)}$ be any collection of field elements, not all zero. Let $a(X,Y_1,...,Y_r)$ stand for the linear combination:
      $$
      a(X,Y_1,...,Y_r) = \sum_{i \in [m], I \subseteq [r], |I| \le \Delta(g_i)} \alpha_{i,I} g_i(X+Y_I)
      $$

      Let $G' \subseteq G$ be the set of all $g_i$'s which appear in $a$, i.e.:
      $$
      G' = \{g_i \in G: \exists I\ \alpha_{i,I} \neq 0\}
      $$

      There \underline{does not} exist polynomials $h_1,...,h_l \in Der(G')$, $l \le \GF(m)$ s.t. $a(X,Y_1,...,Y_r)$ can be expressed as:
      $$
      H(h_1(X+Y_{I_1}),...,h_l(X+Y_{I_l}))
      $$
      for $I_1,...,I_l \subseteq [r]$ and some function $H:\F^l \to \F$.
\end{enumerate}

If the set $G$ satisfies only $(1)$ and $(2)$, we say $G$ is {\em
pre-strong-regular} (notice that $\GF$ appears only in $(3)$).
\end{dfn}

We first prove, similar to the proof in \cite{gt07}, that any set of
polynomials can be refined to a strong $\GF$-regular set, where the
size of the resulting set depends only on the size of the original
set, and the maximal degree of polynomials in it. Also, the refining
set is contained in the
space of iterated derivatives of the original polynomials.

We now formally define the space of iterated derivatives.
\begin{dfn}[Space of iterated derivatives]
For a polynomial set $F$, we define its iterated derivative set
$Der_C$ to be the set of taking at most $C$ derivatives of $F$, i.e.
$$
Der_0(F) = F
$$
$$
Der_C(F) = Der(Der_{C-1}(F)) \cup Der_{C-1}(F)
$$
\end{dfn}

\begin{lemma}[Strong-Regularity Lemma]\label{lemma:regularity}
Let $\GF$ be any growth function. Let $F=\{f_1,...,f_s\}$ be a set
of polynomials of maximal degree $k$. There exist a refinement
$G=\{g_1,...,g_m\}$ of $F$ s.t.
\begin{enumerate}
  \item The maximal degree of polynomials in $G$ is also at most $k$
  \item The set $G$ is strong $\GF$-regular.
  \item The size $m$ of $G$ is a function of only $\GF$, $s$ and $k$. Importantly, it
        is independent of $n$.
  \item There exists $C=C(\GF,s,k)$ s.t. $G \subseteq Der_C(F)$
\end{enumerate}
\end{lemma}

\begin{proof}
We will start by defining a pre-strong-regular set $G$ from $F$, and
will keep refining it until we reach a strong $\GF$-regular set. Our
set $G$ will also be in $Der_i(F)$ at the $i$-th iteration. We will
finish by showing that the refinement process must end in a finite
number of steps.

We start by defining $\Delta:F \to \N$ by $\Delta(f_i) = \deg(f_i)$,
and set the initial value of $G$ to be $F$. To show that the initial
$G$ is
pre-strong-regular with the degree bound $\Delta$, observe that for
any $r > \deg(f_i)$, deriving $f_i$ $r$-times yields the zero
polynomial. Thus, if $Y_1,...,Y_r$ are variables, we have the
identity:
$$
f_i(X + Y_{[r]}) = \sum_{I \subsetneq [r]} (-1)^{r-|I|+1} f_i(X +
Y_I)
$$
Since we can do this for any $r > \deg(f_i)$, we can continue and
express $f_i(X+Y_{[r]})$ as a linear combination of $\{f_i(X+Y_I): I
\subseteq [r],\ |I| \le \deg(f_i)\}$. Thus, $G$ is
pre-strong-regular with the degree bound $\Delta$.

We will continue to refine $G$ as long as it is not strong
$\GF$-regular. Assume $G=\{g_1,...,g_m\}$ at some iteration is not
strong-$\GF$-regular. By definition, there is some $r \ge 0$ and
coefficients $\{\alpha_{i,I}\}_{i \in [m], I \subseteq [r], |I| \le
\Delta(g_i)}$ s.t. the linear combination:
$$
a(X,Y_1,...,Y_r) = \sum_{i \in [m], I \subseteq [r], |I| \le
\Delta(g_i)} \alpha_{i,I} g_i(X+Y_I)
$$
can be expressed as a function of $l \le \GF(m)$ polynomials
$h_1,...,h_l \in Der(G')$, where $G'=\{i \in [m]: \exists I\
\alpha_{i,I} \ne 0\}$ is the set of all $g_i$'s participating in the
linear combination.

Let $g_{i_0}$ be a polynomial of maximal degree $k$ in $G'$ and let
$I_0$ be a maximal $I$ in respect to inclusion s.t.
$\alpha_{i_0,I_0} \ne 0$. Notice that we must have that $|I_0| \le
\Delta(g_{i_0})$. We have:
$$
\sum_{i \in [m], I \subseteq [r], |I| \le \Delta(g_i)} \alpha_{i,I}
g_i(X+Y_I) = H(h_1(X+Y_{J_1}),...,h_l(X+Y_{J_l}))
$$
for some function $H:\F^l \to \F$.

Notice first that $\deg(h_i) \le k-1$ for all $i \in [l]$.
Substitute in the expression $Y_i = 0$ for all $i \notin I_0$. We
get that $g_{i_0}(X + Y_{I_0})$ can be
expressed as a function of $\{g_{i_0}(X+Y_J): J \subsetneq I_0\}$,
$\{g_j(X+Y_J): j \neq i,\ J \subseteq I_0,\ |J| \le \Delta(g_j)\}$
and $\{h_j(X+Y_J): J \subseteq I_0,\ |J| \le \deg(h_j)\}$. Thus, if
we add the polynomials $h_1,...,h_l$ to $G$ (and set $\Delta(h_i) =
\deg(h_i)$), we can reduce $\Delta(g_j)$ to $|I_0|-1$. If we reduced
it to zero, we can remove $g_j$ entirely from $G$. The resulting $G$
will be our set for the next iteration.

In order to prove that the refinement process ends after a finite
number of iterations (depending on the initial size of $F$ and its
maximal degree), notice that at each iteration, the sum of
$\Delta(g_i)$ for all $g_i \in G$ with some degree $d'$ reduces by
at least $1$, where the new
polynomials added are all of degree strictly smaller than $d'$, and
their number is bounded (as a function of $\GF$ and the size of $G$
at the beginning of the iteration). So the total number of
iterations is some Ackerman-like function of the initial number of
polynomials, their maximal degree and the growth function $\GF$.
\end{proof}

\section{Almost independence by strong regularity}\label{sec:almost_ind_by_reg}
In the following we prove that strong regularity
induces almost independence/total dependence structure over general
sets of variables. The following lemmas are the main technical
building blocks in the proof of Theorem~\ref{thm:main}.

%

We start by proving a lemma correlating applications of $g_i$ on
sums below the degree bound $\Delta$ to all sums over a set of
variables.

\begin{lemma}\label{lemma:belowdeltainducesall}
Let $G=\{g_1,...,g_m\}$ be a strong-regular set with degree bound
$\Delta$. Let $x,x' \in \F^n$ be two points s.t. $g_i(x)=g_i(x')$
for all $i \in [m]$. Let $y'_1,...,y'_k \in \F^n$ be values for some
$k \ge 1$, and let $Y_1,...,Y_k \in \F^n$ be $k$ random variables.
Then the following two events are equivalent:
\begin{enumerate}
  \item $A = [g_i(x+Y_I)=g_i(x'+y'_I)$ for all $i \in [m]$ and $I \subseteq [k]]$
  \item $B = [g_i(x+Y_I)=g_i(x'+y'_I)$ for all $i \in [m]$ and $I \subseteq [k]$ s.t. $1 \le |I| \le \Delta(g_i)]$
\end{enumerate}
\end{lemma}

\begin{proof}
It is obvious that if $A$ holds then also $B$ holds. Assume that $B$
holds, i.e. that
$$
g_i(x+Y_I)=g_i(x'+y'_i)
$$
for all $i \in [m]$ and $I \subseteq [k]$ s.t. $|I| \le
\Delta(g_i)$.
Take some $I$ s.t. $I > \Delta(g_i)$. We need to show that also
$g_i(x+Y_{I})=g_i(x'+y'_{I})$. Since $|I| > \Delta(g_i)$ we know by
the strong regularity of $G$ that there is a function $F_{i,I}$ s.t.
$$
g_i(X+Y_I) = F_{i,I}\(g_j(X+Y_J):\ j \in [m],\ J \subseteq I,\ |J|
\le \Delta(g_j)\)
$$
By first substituting $X=x$ to compute $g(x+Y_{I})$, and then
substituting $X=x'$ and $Y_j = y'_j$ to compute
$g(x'+y'_{I})$,
and using that both $g_j(x)=g_j(x')$ for all $j \in [m]$ and the
assumption that $B$ holds, we get that also
$g_i(x+Y_{I})=g_i(x'+y'_{I})$.
\end{proof}

We now prove a lemma showing that points which are sum of at most
$\Delta(g_i)$ points for all $g_i$ are simultaneously almost
disjoint, provided that $\GF$ is large enough. Remember that we are
in the process of proving Theorem~\ref{thm:main} for degree $d$ by
induction. Thus, we assume it to hold for all degrees $d'<d$, and in
particular to all linear combinations of
$g_1,...,g_m$.

\begin{lemma}\label{lemma:belowdeltaindependent}
Let $\gamma = \gamma(m)$ be an error term. Let $Y_1,...,Y_k \in
\F^n$ be random variables for some $k \ge 1$. Assume $\GF$ is large
enough (as a function of $\gamma$ and $k$).
Assume $g_1,...,g_m$ are strong $\GF$-regular with degree bound
$\Delta$. For any non-empty $I \subseteq [k]$ let
$x_I \in \F^n$ be some point, and $a^{(I)}=(a^{(I)}_1,...,a^{(I)}_k)
\in \F^k$ s.t.
\begin{itemize}
  \item $a^{(I)}_i \neq 0$ for all $i \in I$
  \item $a^{(I)}_i = 0$ for all $i \notin I$
\end{itemize}
Then the joint distribution of
$$
\(g_i(x_I + \sum_{i \in I}a^{(I)}_i Y_i):\ i \in [m],\ I \subseteq
[k],\ 1 \le |I| \le \Delta(g_i)\)
$$
is $\gamma$-close to the uniform distribution on $\F^{\sum_{i=1}^m
\sum_{j=1}^{\Delta(g_i)} {k \choose j}}$.
\end{lemma}

Before proving Lemma~\ref{lemma:belowdeltaindependent}, we give an
immediate corollary of it and
Lemma~\ref{lemma:belowdeltainducesall}:

\begin{cor}\label{cor:jointprob}
Let $x,x' \in \F^n$ be two points s.t. $g_i(x)=g_i(x')$ for all $i
\in [m]$. Let $y'_1,...,y'_k \in \F^n$ be values for some $k \ge 1$,
and let $Y_1,...,Y_k \in \F^n$ be $k$ random variables. Then
$$
\P \[g_i(x+Y_I)=g_i(x'+y'_I)\ \forall\ i \in [m],\ I \subseteq [k]\]
= |\F|^{-\sum_{i=1}^m \sum_{j=1}^{\Delta(g_i)} {k \choose j}} (1 \pm
\gamma)
$$
\end{cor}

We need the following simple lemma for the proof of
Lemma~\ref{lemma:belowdeltaindependent}. It states that a random
derivative of a biased polynomial is also biased.
\begin{lemma}\label{lemma:derivativebias}
Let $h(Y_1,...,Y_k)$ be a polynomial with bias $\delta$. Let $h'$ be
the derivation of $h$ in variables $Y_1,...,Y_r$ by directions
$Z_1,...,Z_r$, $(r \le k)$ i.e.
$$
h'(Y_1,...,Y_k,Z_1,...,Z_r) = \sum_{w \in \{0,1\}^r} (-1)^{|w|}
h(Y_1 + w_1 Z_1,...,Y_r + w_r Z_r, Y_{r+1},...,Y_k)
$$
where $|w|$ denotes the hamming weight of $w$. Then $bias(h') \ge
\delta^{2^r}$.
\end{lemma}
\begin{proof}
We apply Cauchy-Schwartz. It's enough to prove for $k=2$ and $r=1$
because we can group variables.
\begin{align*}
bias(h') = & \E_{Y_1,Y_2,Z_1 \in
\F^n}[\omega^{h(Y_1,Y_2)-h(Y_1+Z_1,Y_2)}] =
\E_{Y_2 \in \F^n} [\(\E_{Y_1 \in \F^n} [\omega^{h(Y_1,Y_2)}]\)^2] \ge \\
&\(\E_{Y_1,Y_2 \in \F^n} [\omega^{h(Y_1,Y_2)}]\)^2 = \delta^2
\end{align*}
\end{proof}

\begin{proof}(of Lemma~\ref{lemma:belowdeltaindependent})
We start by using the well known fact, that if a distribution over
$\F^r$ is not uniform, it must have some biased functional. If the
distribution we study is
$\gamma$-far from uniform, then there must be a linear functional on
$\{g_i(x_I+\sum_{i \in I}a^{(I)}_i Y_i):\ i \in [m],\ I \subset
[k],\ |I| \le \Delta(g_i)\}$ with some non-negligible bias depending
on $\gamma$. We will prove that if we assume that, we reach a
contradiction.

Denote by
$Y'_I = \sum_{i \in I}a^{(I)}_i Y_i$, and notice it depends on
exactly the same set of variables from $Y_1,...,Y_k$ as
$Y_I$. By our assumption, there exist coefficients
$\{\alpha_{i,I}\}$, not all zero, s.t. the polynomial
$$
h(Y_1,...,Y_k) = \sum_{i \in [m], I \subseteq [k], |I| \le
\Delta(g_i)} \alpha_{i,I} g_i(x_I+Y'_I)
$$
has bias at least $\rho$, where $\rho$ is a function of $\gamma$,
$k$ and $m$ only (and not of $n$).

Fix $I_0$ maximal with regards to inclusion s.t. not all
$\alpha_{i,I_0}$ are zero. Since we just care about the bias of $h$
under random $Y_1,...,Y_k$, we can multiply each $Y_i$ by some
non-zero coefficient. We thus assume w.l.o.g that $a^{(I_0)}_i = 1$
for all $i \in I_0$.
Let $|I_0|=r$. We assume w.l.o.g that $I_0 = \{1,2,...,r\}$. Notice
that $Y'_{[r]}=Y_{[r]}$. We also shorthand $x=x_{[r]}$.

Let $g_{i_0}$ be a polynomial with maximal degree
$d'' \le d' < d$ s.t. $\alpha_{i_0,I_0} \ne 0$.

We derive now once each of the variables in $Y_1,...,Y_r$. Let
$\{Z_i\}_{i=1..r}$ be new variables in $\F^n$,

and consider:
$$
h'(Y_1,...,Y_k,Z_1,...,Z_r) = \sum_{w \in \{0,1\}^r} (-1)^{|w|}
h(Y_1 + w_1 Z_1,...,Y_r + w_r Z_r,Y_{r+1},...,Y_k)
$$
First, by Lemma~\ref{lemma:derivativebias}, $h'$ has bias at least
$\rho'=\rho^{2^k}$.

Now, consider what happens to a term
$g_i(x+Y'_I)$ in $h$ after the derivation. If $I \neq [r]$, by the
maximality of $I_0$ there must exist $i' \in [r]$ s.t. $i' \notin
I$. Thus, deriving $Y_{i'}$ zeroes out
$g_i(x+Y'_I)$.

So, the only terms remaining in $h'$ come from terms in $h$ of the
form $g_i(x+Y_{[r]})$. Thus, $h'$ does not depend on $Y_i$ for $i
\notin [r]$, and also all the $g_i$'s remaining must have
$\Delta(g_i) \ge r$ (because $g_i(x+Y_{[r]})$ appeared in $g$ with
non-zero coefficient). Thus we can write:
$$
h'=h'(Y_1,...,Y_r,Z_1,...,Z_r) = \sum_{i \in [m]} \alpha_{i,[r]}
\sum_{w \subseteq [r]} (-1)^{|w|} g_i(x+Y_{[r]} + Z_w)
$$

We now make an important observation. Notice that $h'$ depends only
on the sum $Y_{[r]}$, and not on the individual $Y_1,...,Y_r$. So we
can substitute $W = x + Y_{[r]}$ and get:
$$
h'=h'(W,Z_1,...,Z_r) = \sum_{i \in [m]} \alpha_{i,[r]} \sum_{w
\subseteq [r]} (-1)^{|w|} g_i(W+Z_w)
$$

We have assumed that $G$ is
strong $\GF$-regular. We will show now that if we choose $\GF$ large
enough, we have already reached a contradiction. Notice the
polynomials $g_i(W+Z_w)$ are exactly those which appear in the
regularity requirements (
where $X$ is replaced here by $W$, and $Y_1,Y_2,...$ by
$Z_1,Z_2,...$). Let $G'$ denote the set of $g_i$'s s.t. $g_i$ appear
in $h'$ with non-zero coefficient.

We assume by induction that Theorem~\ref{thm:main} holds for $d''<d$
and for all $n$. Since all polynomials $g_i \in G$ have degree at
most $d-1$, then also $\deg(h') \le d-1$, and so we can apply
Theorem~\ref{thm:main} on $h'$. So, since $h'$ has bias $\rho'$,
there must exist polynomials $q_1,...,q_t \in Der(h')$ s.t.
$$
h'(W,Z_1,...,Z_r) = Q(q_1(W,Z_1,...,Z_r),...,q_t(W,Z_1,...,Z_r))
$$
for some function $Q:\F^t \to \F$, s.t. $t=t(\rho',d'')$. Moreover,
since every polynomial $q_i$ is of the form
$h'(W+a_0,Z_1+a_1,...,Z_r+a_r) - h'(W,Z_1,...,Z_r)$ for some
constants $a_0,...,a_r \in \F^n$, and $h'$ is the sum of $g_i(W +
Z_w)$, we can decompose each $q_i$ to a sum of at most $2^r$
polynomials of the form $g_i(W + Z_w + a) - g_i(W + Z_w) \in
Der(G')$ for $w \subseteq \{0,1\}^r$. Let $q'_1,...,q'_{t'} \in
Der(G')$ denote these decomposed polynomials. We thus have that:
$$
h'(W,Z_1,...,Z_r) = Q'(q'_1(W + Z_{I'_1}),...,q'_{t'}(W +
Z_{I'_{t'}}))
$$
for some function $Q':\F^{t'} \to \F$, $t' = 2^r t$ and
$I'_1,...,I'_t \subseteq [r]$. We got that we can compute
$$
h'(W,Z_1,...,Z_r) = \sum_{i \in [m]} \alpha_{i,[r]} \sum_{w
\subseteq [r]} (-1)^{|w|} g_i(W+Z_w)
$$
as a function of $t'$ polynomials of degree strictly smaller than
$d''$. If we have $\GF(m) > t'$ this is a contradiction
to the strong $\GF$-regularity of $g_1,...,g_m$.

Summarizing, there can be no linear combination of $\{g_i(x + Y_I):
I \in S, 1 \le |I| \le \Delta(g_i)\}$ which has bias more than
$\rho$, and so the distribution is $\gamma$-close to uniform.
\end{proof}

\section{From approximation to calculation: proof of Theorem~\ref{thm:main}}\label{sec:approxtocalc}
In this section we prove Theorem~\ref{thm:main}. The main technical
tool that we will use are Lemmas~\ref{lemma:belowdeltainducesall}
and~\ref{lemma:belowdeltaindependent}. Let $p(X)$ stand for a degree
$d$ polynomial with bias $\delta$. The proof of the theorem is
immediate given the following two lemmas. The first lemma (Lemma
\ref{lemma:approximation}) asserts that a biased degree $d$
polynomial can be approximated by constant many degree $d-1$
polynomials. This lemma was useful also in the proof of Green and
Tao and its proof appears in~\cite{bv}. The second lemma (Lemma
\ref{lemma:computation}) asserts that approximation by few degree
$d-1$ polynomials imply computation by few degree $d-1$ polynomials.
In the following we present the two lemmas.


\begin{lemma}[Bias imply approximation by few lower degree polynomials]\label{lemma:approximation}
Let $p(X)$ be a polynomial of degree $d$ with bias $\delta$. For any
$\epsilon>0$ there exist polynomials $f_1(X),...,f_s(X)$ of degree
at most $d-1$ and a function $F:\F^s \to \F$ s.t.
$$
\P_{X \in \F^n}[F(f_1(X),...,f_s(X)) \ne p(X)] < \epsilon
$$
The number $s$ of the polynomials depends only on $\delta$ and
$\epsilon$. Moreover, $f_1,...,f_s \in Der(p)$.
\end{lemma}

The full proof can be found in \cite{bv} (Lemma 24). The proof idea
is that a random derivative in direction $a$ approximates $p(x)$ for
any $x$, and so taking a majority value over enough random values of
$a$'s (but still a constant number) allows to compute $p$ on all but
a $\epsilon$-fraction of the points.

\begin{lemma}[Approximation by few lower degree polynomials imply computation by few lower degree polynomials]\label{lemma:computation}
Let $p(X)$ be a polynomial of degree $d$, $f_1,...,f_s$ polynomials
of degree $d-1$, ($s = O(1)$) and $H:\F^s \to \F$ a function s.t.
the composition $H(f_1(X),...,f_s(X))$ $\epsilon_d$-approximates
$p$, where $\epsilon_d = 2^{-\Omega(d)}$ Then there exist $s'$
polynomials $f'_1,...,f'_{s'}$ and a function $H':\F^{s'} \to \F$
s.t.
$$
H'(f'_1(X),...,f'_{s'}(X)) \equiv p(X)
$$
Moreover, $s'=s'(d,s)$ (i.e. independent of $n$) and each $f'_i$ if
of the form $p(X+a)-p(X)$ or $f_j(X+a)$ for $a \in \F^n$.
\end{lemma}

Thus, to complete the proof of Theorem~\ref{thm:main}, it remains to
prove Lemma \ref{lemma:computation}.

We start the proof of Lemma~\ref{lemma:computation} by refining
$F=\{f_1,...,f_s\}$ to a strong-regular set. Let $\GF$ be a large
enough growth function (to be determined later). By
Lemma~\ref{lemma:regularity} there exists a set $G=\{g_1,...,g_m\}$
refining $F$, and a degree bound $\Delta$, s.t. $G$ is strong
$\GF$-regular with degree bound $\Delta$. Moreover, there exists a
$C=C(\GF,\delta,d)$ s.t. $G \subseteq Der_C(F)$. We know that $G$
also approximates $p(X)$ at least as well as $F$ does. We will prove
that it is in fact computes $F$ completely. We can then decompose
each $g_i \in Der_C(F)$ as a sum of at most $2^C$ elements in
$Der(p)$ to conclude the result.

Thus, we need to show that $G$ in fact computes $p(X)$ completely.
For $c=(c_1,...,c_m) \in \F^m$, denote by $R_c \subseteq \F^n$ the
region
$$
R_c = \{x \in \F^n: g_i(x) = c_i \}
$$
To show that $G$ computes $p(X)$ is equivalent to showing that
$p(X)$ is constant on any region $R_c$. Thus, we turn to study the
regions $R_c$.

We first show (Lemma~\ref{lemma:regionsuniform}) that all regions
$R_c$ have about the same volume, i.e. that they form an almost
uniform division of $\F^n$ to $\F^m$ regions.
Since $G$ is a strong regular refitment of $F$ that
$\epsilon_d$-approximates $p$ we know that also $G$
$\epsilon_d$-approximates $p$, i.e. there exists some $H':\F^m \to
\F$ s.t.
$$
\P_{X \in \F^n}[H'(g_1(X),...,g_m(X)) \ne p(X)] < \epsilon_d
$$

For every region $R_c$, let $\eta_c$ be the probability that $p$ is
different from $G$ on that region
($G$ is constant on the region).
$$
\eta_c = \P_{X \in R_c}[p(X) \ne G|_{R_c}]
$$
Since the average of $\eta_c$ is at most $\epsilon_d$, and all
regions are almost uniform (Lemma~\ref{lemma:regionsuniform}) there
can be at most $\sqrt{\epsilon_d} |\F|^m$ regions on which $\eta_c >
\sqrt{\epsilon_d}$. We call these the {\em bad regions}, and we call
the rest of the regions {\em almost good regions}. Next we show
(Lemma~\ref{lemma:noalmostconstantregions}) that the almost good
regions are totally good and $p$ is fixed on them. Last, we use the
fact that there are only few bad regions and $p$ is fixed on the
rest to conclude that $p$ is also fixed on the bad regions
(Lemma~\ref{lemma:allregionsconstant}). Thus, $p(X)$ is in fact
constant on all regions. To complete the proof of
Lemma~\ref{lemma:computation}, it remains to prove
Lemmas~\ref{lemma:regionsuniform},~\ref{lemma:noalmostconstantregions}
and~\ref{lemma:allregionsconstant}.

\begin{lemma}[Regions are uniform]\label{lemma:regionsuniform}
Let $\gamma = \gamma(m)>0$ be a small enough error term. If $\GF$ is
large enough than
$$
|R_c| = |\F|^{n-m} (1 \pm \gamma)
$$
for all $c \in \F^m$.
\end{lemma}

\begin{proof}
Let $c \in \F^m$ and assume first that $R_c$ is not empty, i.e.
there exist some $x$ s.t. $g_i(x)=c_i$
for all $i \in [m]$. We apply Corollary~\ref{cor:jointprob} with
$k=1$, $x'=x$ and $y_1=0$ and get:
$$
\P_{Y_1}[g_i(x+Y_1)=g_i(x),\ \forall\ i \in [m]] = |\F|^{-m} (1 \pm
\gamma)
$$
Substituting $Y=x+Y_1$ proves the result for $R_c$.

To show the there can be no empty regions, assume otherwise. Thus,
there are at most $|\F|^m-1$ non-empty cells, and each has volume at
most $|\F|^{n-m}(1+\gamma)$. Thus $(|\F|^m - 1)|\F|^{n-m} (1+\gamma)
\ge |\F|^n$. If $\gamma(m) < |\F|^{-m}$ we get a contradiction.
Thus, there are no empty regions, and so all regions have volume
$|\F|^{n-m} (1 \pm \gamma)$.
\end{proof}

\begin{lemma}[Almost good regions are good]\label{lemma:noalmostconstantregions}
Let $R_c$ be a region s.t
$$
\P_{X \in R_c}[p(X) = b] > 1 - 2^{-2(d+1)}
$$
for some constant $b \in \F$. Then $p(X)=b$ for all $X \in R_c$.
\end{lemma}

Before proving the lemma we need the following counting lemma on the
number of hypercubes and pairs of hypercubes inside a region,
similar to one in \cite{gt07}. However, our technique avoids the
need of integration.

\begin{lemma}\label{lemma:counting}
Let $\gamma=\gamma(m)>0$ be small enough error term, and assume
$\GF$ is large enough. For any point $R = R_c$ and a point $x \in R$
we have:
\begin{enumerate}
  \item Let $Y_1,...,Y_{d+1}$ be variables in $\F^n$. Then:
  \begin{align*}
  \P_{Y_1,...,Y_{d+1} \in \F^n}[x + Y_I \in R,\ \forall I \subseteq [d+1]] = |\F|^{- \sum_{i=1}^m \sum_{j=1}^{\Delta(g_i)} {d+1 \choose j}} (1 \pm \gamma)
  \end{align*}

  \item Let $Y_1,...,Y_{d+1},Z_1,...,Z_{d+1}$ be variables in $\F^n$. For any non-empty $I_0 \in [d+1]$:
  \begin{align*}
  &\P_{Y_1,...,Y_{d+1},Z_1,...,Z_{d+1} \in \F^n}\[x + Y_I \in R, x + Z_I \in R,\ \forall I \subseteq [d+1] |Y_{I_0} = Z_{I_0}\] \le \\
  &|\F|^m \(|\F|^{- \sum_{i=1}^m \sum_{j=1}^{\Delta(g_i)} {d+1 \choose j}}\)^2 (1 + \gamma)
  \end{align*}
\end{enumerate}
\end{lemma}

\begin{proof}

\begin{enumerate}
\item
This is a direct application of Corollary~\ref{cor:jointprob} for
$k=d+1$, $x'=x$ and $y_1,...,y_k = 0$.

\item
Assume w.l.o.g that $I_0 = \{1,2,...,s\}$ for $1 \le s \le d+1$. We
start by making a linear transformation on the coordinates to bring
$Y_{I_0}$ and $Z_{I_0}$ to a single variable. Let $Y'_i=Y_i$ for $i
\ne s$ and $Y'_s = Y_1+...+Y_s$, and similarly define
$Z'_1,...,Z'_{d+1}$. We write $Y_I$ in the basis of
$Y'_1,...,Y'_{d+1}$. Divide $I = I_s \cup I_{\bar{s}}$ where $I_s =
I \cap [s]$ and $I_{\bar{s}}= I \setminus I_s$. We have:

\begin{itemize}
  \item If $s \notin I$, $Y_I = \sum_{i \in I} Y'_i$
  \item If $s \in I$, $Y_I = Y'_s - \sum_{i \in [s] \setminus I_s} Y'_i + \sum_{i \in I_{\bar{s}}} Y'_i$
\end{itemize}

Consider for every $I$ the set $T_I$ of indices of
$Y'_i$ which appear in the expansion of $Y_I$. Notice that for any
$T \subseteq [d+1]$ there is exactly one $I$ s.t. $T_I = T$. In
particular, in order that $g_i(x+Y_I)=g_i(x)$ for all $I$, we must
have in particular that:
\begin{itemize}
  \item For any $I \subseteq [d+1]$ s.t. $s \notin I$ and $|I| \le \Delta(g_i)$,
  $$
  g_i(x + Y'_I) = g_i(x)
  $$
  \item For any $I \subseteq [d+1]$ s.t. $s \in I$ and $|I| \le \Delta(g_i)$,
  $$
  g_i(x + Y'_s - Y'_{I \cap [s-1]} + Y'_{I \cap \{s+1,...,d+1\}}) = g_i(x)
  $$
\end{itemize}

Similarly for the $Z'$'s, using the fact that the event
$Y_{I_0}=Z_{I_0}$ translates to $Z'_s=Y'_s$:
\begin{itemize}
  \item For any $I \subseteq [d+1]$ s.t. $s \notin I$ and $|I| \le \Delta(g_i)$,
  $$
  g_i(x + Z'_I) = g_i(x)
  $$
  \item For any $I \subseteq [d+1]$ s.t. $s \in I$ and $|I| \le \Delta(g_i)$,
  $$
  g_i(x + Y'_s - Z'_{I \cap [s-1]} + Z'_{I \cap \{s+1,...,d+1\}}) = g_i(x)
  $$
\end{itemize}

The probability of this event is an upper bound on our required
probability. Since our variables
$$
Y'_1,...,Y'_{d+1},Z'_1,..,Z'_{s-1},Z'_{s+1},...,Z'_{d+1}
$$
are uniform and independent, we can apply
Lemma~\ref{lemma:belowdeltaindependent} to show that its probability
is the required upper bound.
The number of subsets of size $j > 1$ in the above events is ${d+1
\choose j}$ for the event on the $Y'$'s, and also ${d+1 \choose j}$
for the event on $Z'_1,...,Z'_{s-1},Y'_s,Z'_{s+1},...,Z'_{d+1}$. For
$j=1$ however we have intersection ($Y'_s$ is appearing twice), and
so the number of events is $2{d+1 \choose 1}-1$. Thus,by
Lemma~\ref{lemma:belowdeltaindependent} the probability of the total
event is:
$$
|\F|^m \(|\F|^{-\sum_{i=1}^m \sum_{j=1}^{\Delta(g_i)} {d+1 \choose
j}}\)^2 (1 \pm \gamma)
$$
which upper bounds the required probability.

\end{enumerate}
\end{proof}

We now prove Lemma~\ref{lemma:noalmostconstantregions} using
Lemma~\ref{lemma:counting}. We follow the same proof as in
\cite{gt07}.

\begin{proof}
Let $B \subseteq R$ be the set of all "bad" points $x \in R$ on
which $p(x) \ne b$. By our assumption, $|B| < 2^{-2(d+1)} |R|$.
Assume $B$ is non-empty, and choose some $x \in B$. Let
$Y_1,...,Y_{d+1}$ be random variables in $\F^n$. Fix small enough
$\gamma=\gamma(m)$. By Lemma~\ref{lemma:counting} (1),
$$
p_R = \P[x + Y_I \in R,\ \forall I \subseteq [d+1]] \ge |\F|^{-
\sum_{i=1}^m \sum_{j=1}^{\Delta(g_i)} {d+1 \choose j}} (1 - \gamma)
$$

We now wish to bound the event that when all $X + Y_I$ are in $R$,
some $X + Y_I$ is in $B$, and then union bound over all possible
$I$.

We start by applying Cauchy-Schwartz to transform the problem to
counting pairs of hypercubes. Fix some non-empty $I_0 \subseteq
[d+1]$, and let
\begin{align*}
p_B =& \P[x + Y_I \in R\ \forall I \subseteq [d+1]\ \wedge\ x +
Y_{I_0}
\in B] = \\
&\sum_{x_0 \in B} \P[x + Y_I \in R\ \forall I \subseteq [d+1]\
\wedge\ x + Y_{I_0}=x_0]
\end{align*}

We need to upper bound $p_B$.
\begin{align*}
p_B^2 = & \(\sum_{x_0 \in B} \P[x + Y_I \in R\ \forall I \subseteq [d+1]\ \wedge\ x + Y_{I_0}=x_0]\)^2 \le \\
& |B| \sum_{x_0 \in B} \P[x + Y_I \in R\ \forall I \subseteq [d+1]\ \wedge\ x + Y_{I_0}=x_0]^2 = \\
& |B| \P[x + Y_I \in R\ \forall I \subseteq [d+1]\ \wedge\ x + Z_I \in R\ \forall I \subseteq [d+1]\ \wedge\ x + Y_{I_0} = x + Z_{I_0}] = \\
& |B| |\F|^{-n} \P[x + Y_I \in R,\ x + Z_I \in R\ \forall I
\subseteq [d+1]|x + Y_{I_0} = x + Z_{I_0}]
\end{align*}
where $Z_1,...,Z_{d+1}$ are new variables in $\F^n$.

By claim (2) in Lemma~\ref{lemma:counting} we get that this
probability is at most
$$
|B| |\F|^{m-n} p_R^2 (1 + \gamma)
$$

By
Lemma~\ref{lemma:regionsuniform}, $|R| = |\F|^n \P_{X \in \F^n}[X
\in R] = |\F|^{n-m} (1 \pm \gamma)$. Thus, we have that:
$$
p_B^2 \le \frac{|B|}{|R|} p_R^2 (1 \pm 2\gamma) \le 2^{-2(d+1)}
p_R^2
$$

and thus $\frac{p_B}{p_R} \le 2^{-(d+1)} (1 \pm 2 \gamma)$. We can
now union bound over all non-empty $I_0 \subseteq [d+1]$. The
probability that there is some $I_0$ for which $x + Y_{I_0} \in B$
is at most
$$
(2^{d+1}-1) (2^{-(d+1)} + \gamma) < 1
$$
for small enough $\gamma$.

Thus, there must exist $y_1,...,y_{d+1} \in \F^n$ s.t.
$$
x + y_I \in R \setminus B
$$

for all non-empty $I \subseteq [d+1]$. Equivalently, $p(x + y_I) =
b$ for all such $I$'s. However, since $p(X)_{y_1,...,y_{d+1}} \equiv
0$,
$$
p(x) = \sum_{I \subseteq [d+1], |I|>0} (-1)^{|I|+1} p(x + y_I)
$$
and so if all $p(x+y_I)=b$, then also $p(x)=b$, hence $x \notin B$.
So we have proved that $B$ is empty, i.e. $p$ is constant on $R$.
\end{proof}

We finish the proof of Theorem~\ref{thm:main} by proving that if
$p(X)$ is constant over almost all regions, then it must be constant
over any region.

\begin{lemma}[If almost all regions are totally good, all are totally good]\label{lemma:allregionsconstant}
Assume that the fraction of regions on which $p$ is constant is at
least
$1-2^{-(d+2)}$. Then $p$ is constant over any region.
\end{lemma}

\begin{proof}
Let $R$ be any region, and $x,x' \in R$ two points in $R$. We need
to show that $p(x)=p(x')$. Choose $y'_1,...,y'_{d+1} \in \F^n$
randomly. The probability that $x' + y'_I$ falls in a bad region for
any non-empty $I \subseteq [d+1]$ is $2^{-(d+2)}$
(since regions are almost uniform, see
Lemma~\ref{lemma:regionsuniform}). Thus, applying union bound over
all non-empty $I \subseteq [d+1]$ we get that $\{x'+y'_I\}$ fall in
good regions for all non-empty $I$ with probability at least $1/2$.
Fix some $y'_1,...,y'_{d+1}$ fulfilling this requirement.

Let $Y_1,...,Y_{d+1} \in \F^n$ be random variables. Since
$g_i(x)=g_i(x')$ for all $i \in [m]$ we can apply
Corollary~\ref{cor:jointprob}:
$$
\P\[g_i(x+Y_I)=g_i(x'+y'_I)\ \forall\ i \in [m],\ I \subseteq
[d+1]\] = |\F|^{-\sum_{i=1}^m \sum_{j=1}^{\Delta(g_i)} {d+1 \choose
j}} (1 \pm \gamma)
$$
In particular, for small enough $\gamma$ we get that
$$
\P\[g_i(x+Y_I)=g_i(x'+y'_I)\ \forall\ i \in [m],\ I \subseteq
[d+1]\]
> 0
$$
Let $y_1,...,y_{d+1}$ be such assignment to $Y_1,...,Y_{d+1}$. We
thus have that for all non-empty $I \subseteq [d+1]$ and for all $i
\in [m]$, $g_i(x+y_I) = g_i(x'+y'_I)$. Since the region of $x'+y'_I$
is good for all non-empty $I$, we get that for all non-empty $I
\subseteq [d+1]$,
$$
p(x+y_I) = p(x'+y'_I)
$$
We now use the fact that $p$ is a degree d polynomial. If we derive
$p$ $d+1$-times in any direction, we will always get zero. We thus
have that
for $x,y_1,...,y_{d+1} \in \F^n$:
$$
\sum_{I \subseteq [d+1]} (-1)^{|I|} p(x+y_I) = 0
$$
Since the same identity is true for $x',y'_1,...,y'_{d+1}$, we get
that $p(x)=p(x')$.
\end{proof}

\textbf{Acknowledgement}
We would like to thank Avi Wigderson, Noga Alon and Terrence Tao for helpful discussions.
The second author would like to thank his advisor, Omer Reingold, for his help and support.


\begin{thebibliography}{99}

\bibitem[GT07]{gt07} B. Green, T.Tao,The distribution of polynomials over finite
fields, with applications to the Gowers norms, preprint, 2007.

\bibitem[BV]{bv} A. Bogdanov and E. Viola.
Pseudorandom bits for polynomials via the Gowers norm. In {\em the
48th Annual Symposium on Foundations of Computer Science (FOCS
2007)}.

\bibitem[MS]{ms} J. MacWilliams and N. J. A. Sloane, {\em The
Theory of Error Correcting Codes}, Amsterdam, North-Holland, 1977.

\bibitem[LMS]{lms} S. Lovett, R. Meshulam and A. Samorodnitsky,
The Inverse Conjecture for the Gowers Norm is False, to appear in
{\em the 40th ACM Symposium on Theory of Computing (STOC 2008)}.

\bibitem[DLMORSW]{r} Ilias Diakonikolas, Homin K. Lee, Kevin Matulef, Krzysztof Onak, Ronitt Rubinfeld, Rocco A. Servedio, Andrew Wan
Testing for Concise Representations. {\em the 48th Annual Symposium
on Foundations of Computer Science (FOCS 2007)}.

\bibitem[RS]{rs}, Ronitt Rubinfeld and Madhu Sudan,
Robust characterizations of polynomials with applications to program
testing, {\em SIAM Journal on Computing,  25 (2), 252--271, 1996}.

\end{thebibliography}
\end{document}